\documentclass[final,3p,times]{elsarticle}
\usepackage{amssymb}
\usepackage{amsmath}
\usepackage{color}
\usepackage{graphicx}
\graphicspath{{images/}}
\usepackage{amssymb}
\usepackage{diagbox}
\usepackage{subcaption}
\usepackage{algorithm}
\usepackage{algpseudocode}
\usepackage{enumitem}
\usepackage{lipsum}
\usepackage{caption}
\biboptions{sort&compress}

\makeatletter
\def\ps@pprintTitle{%
    \let\@oddhead\@empty
    \let\@evenhead\@empty
    \let\@oddfoot\@empty
    \let\@evenfoot\@oddfoot}
\makeatother

\usepackage{mathptmx}  
\usepackage{hyperref}
\hypersetup{hypertex=true,
colorlinks=true,
linkcolor=blue,
anchorcolor=blue,
citecolor=blue}
\usepackage{amsthm}
\newtheorem{theorem}{Theorem}[section]
\newtheorem{definition}{Definition}[section]   % 定义
\newtheorem{lemma}{Lemma}[section]      % 引理
\newtheorem{assumption}{Assumption}[section]

\newtheorem{remark}{Remark}[section]
\journal{Communications in Nonlinear Science and Numerical Simulation}

\begin{document}

\begin{frontmatter}



\title{Incorporating Local H\"older Regularity into PINNs for Solving Elliptic PDEs
}


\author[inst1]{Qirui Zhou}
\author[inst1]{Jiebao Sun}
\author[inst1]{Yi Ran\corref{corresponding}}
% \cormark[1]
\ead{21b912025@stu.hit.edu.cn}
\author[inst1]{Boying Wu}

\cortext[corresponding]{Corresponding author.}
% \cormark[1]
% \ead{z.edu.cn}

\affiliation[inst1]{organization={School of Mathematics},
            addressline={Harbin Institute of Technology}, 
            city={Harin},
            postcode={150001}, 
            % state={State One},
            country={China}}

% \author[inst2]{Author Two}
% \author[inst1,inst2]{Author Three}

% \affiliation[inst2]{organization={Department Two},%Department and Organization
%             addressline={Address Two}, 
%             city={City Two},
%             postcode={22222}, 
%             state={State Two},
%             country={Country Two}}

\begin{abstract}
In this paper, local H\"older regularization is incorporated into a physics-informed neural networks (PINNs) framework for solving elliptic partial differential equations (PDEs). Motivated by the interior regularity properties of linear elliptic PDEs, a modified loss function is constructed by introducing local H\"older regularization term. To approximate this term effectively, a variable-distance discrete sampling strategy is developed. Error estimates are established to assess the generalization performance of the proposed method. Numerical experiments on a range of elliptic problems demonstrate notable improvements in both prediction accuracy and robustness compared to standard physics-informed neural networks.
\end{abstract}


% \begin{highlights}
% \end{highlights}

\begin{keyword}
Physics-informed neural networks\sep Local H\"older regularization\sep Elliptic partial differential equations\sep Numerical simulation
% %% MSC codes here, in the form: \MSC code \sep code
% %% or \MSC[2008] code \sep code (2000 is the default)
\end{keyword}

\end{frontmatter}

\section{Introduction}
\label{intro}
Deep neural networks (DNNs) have recently shown great potential for solving partial differential equations (PDEs), and the theoretical foundation lies in the Universal Approximation Theorem \cite{cybenko1989approximation,hornik1989multilayer} guarantees the solvability of PDEs. Numerous deep neural network-based methods have been proposed for solving PDEs \cite{raissi2019physics,sirignano2018dgm,weinan2018deep,kharazmi2019variational}, each based on different formulations or representations of the equations.

A prominent category of deep learning methods for solving PDEs is built upon their strong form, where the residuals of the governing equations are directly enforced.
Among these, the Physics-Informed Neural Networks (PINNs) framework proposed by Raissi et al. \cite{raissi2019physics} has gained considerable attention.
This approach embeds the physical laws governing the system into the loss function, enabling training without requiring labeled data for the solution. This formulation provides a flexible framework for solving forward and inverse problems in a data-efficient manner. Numerous extensions to the PINNs framework have been proposed to enhance optimization and scalability.
For example, adaptive activation functions have been introduced to improve training stability \cite{jagtap2020adaptive}, and gradient-informed objectives have been shown to accelerate convergence \cite{yu2022gradient}.
To address the computational bottleneck in long-time integration of time-dependent PDEs, Meng et al. \cite{meng2020ppinn} proposed the Parareal PINN (PPINN), which decomposes the temporal domain into parallelizable short-time subproblems guided by a coarse-grained solver.
In addition, libraries such as DeepXDE \cite{lu2021deepxde} provide a unified and extensible interface for implementing such physics-informed models. Another representative strong-form method is the Deep Galerkin Method (DGM) \cite{sirignano2018dgm}, which formulates the loss function as the integral of squared PDEs residuals approximated by Monte Carlo sampling.

Another representative class of approaches leverages the weak formulation of PDEs. The Deep Ritz Method, introduced by E et al. \cite{weinan2018deep}, minimizes the energy functional associated with the target equation, offering a natural fit for elliptic problems. Extending the variational framework, Variational Physics-Informed Neural Networks (VPINNs) \cite{kharazmi2019variational} and their refinement hp-VPINNs \cite{kharazmi2021hp} adopt the Petrov–Galerkin formulation, allowing for localized basis functions and adaptive resolution. Furthermore, the Weak Adversarial Network (WAN) \cite{zang2020weak} recasts the weak form as a min–max optimization and solves it via adversarial learning, enabling unified treatment of forward and inverse problems. These approaches demonstrate the capacity of weak formulations to incorporate diverse boundary conditions and complex solution structures.

The strong-form and weak-form methods are typically designed for solving individual PDEs instances, which limits their ability to generalize across structurally similar PDEs. To overcome this limitation, operator learning has emerged as a powerful paradigm that aims to learn mappings between entire function spaces, such as parameter spaces to solution spaces, thereby enabling the efficient solution of whole classes of PDEs sharing common structures. Numerous neural architectures have been proposed to realize operator learning, including DeepONet \cite{lu2021learning}, Fourier Neural Operator (FNO) \cite{li2020fourier}, U-shaped neural operators (U-NO) \cite{rahman2022u}, and U-FNO \cite{wen2022u}. 
These methods leverage data-driven training to capture complex nonlinear mappings, offering significant speed-ups in evaluating PDEs solutions for new inputs compared to traditional solvers. Despite these advantages, operator learning methods still face challenges such as relatively slow convergence during training and considerable computational costs.

Despite their remarkable performance, DNNs remain highly vulnerable to adversarial perturbations—subtle changes to the input that can cause the model to make incorrect predictions with high confidence \cite{goodfellow2014explaining}. A widely adopted strategy to improve robustness is adversarial training, which incorporates adversarial examples into the training process to enhance model stability \cite{madry2017towards, shafahi2019adversarial}. While this approach has demonstrated effectiveness, it often comes at the cost of increased computational burden and a tendency to overfit specific types of attacks.

Alternatively, regularization-based methods aim to improve robustness by constraining the capacity of neural networks, thereby preventing them from overfitting to adversarial noise or spurious patterns. A common strategy is to impose norm-based penalties on network parameters to promote sparsity or weight decay, which can lead to smoother decision boundaries \cite{gouk2021regularisation}. Beyond standard weight regularization, several works have explored more targeted techniques, including Jacobian regularization \cite{hoffman2019robust, jakubovitz2018improving}, which penalizes the sensitivity of the model output with respect to input perturbations, and Lipschitz regularization \cite{tsuzuku2018lipschitz, miyato2018spectral}, which enforces constraints on the Lipschitz constant of the model to limit output variation under input changes. Despite their conceptual simplicity, regularization-based approaches face several practical challenges. Designing effective regularization terms often requires a deep understanding of the underlying data and model behavior, while tuning the associated hyperparameters can be computationally intensive and highly task-dependent. Moreover, overly aggressive regularization may impair the expressiveness of the model, leading to underfitting and degraded performance on clean or mildly perturbed inputs. These trade-offs highlight the need for carefully balanced regularization schemes tailored to specific robustness objectives.

To address the limitations of existing regularization methods, such as lack of interpretability, limited theoretical support, and sensitivity to hyperparameter choices, approaches that are both mathematically rigorous and practically effective are increasingly being needed. Motivated by the interior regularity of weak solutions to linear elliptic partial differential equations, a novel regularization strategy is proposed based on the local H\"older continuity of the network output. Within this framework, the local H\"older seminorm is rigorously defined on neighborhoods centered at selected training points to quantitatively characterize spatial smoothness. To efficiently and accurately approximate the seminorm, a variable-distance sampling scheme is introduced. Furthermore, the framework supports flexible sampling strategies and hyperparameter tuning, enhancing robustness and adaptability in practical applications. In summary, this strategy bridges classical PDEs regularity theory for weak solutions with modern neural PDEs solvers, providing a theoretically principled and practically effective regularization paradigm.

The rest of this paper is organized as follows. In Section \ref{Preliminaries}, the interior regularity theory and the H\"older norm for divergence-form elliptic partial differential equations are reviewed. Section \ref{Methodology} introduces the proposed framework based on the local H\"older seminorm, with an emphasis on a variable-distance sampling scheme for discrete seminorm approximation. The algorithm and error estimate are also established in Section \ref{Methodology}. Numerical experiments validating the effectiveness of the proposed model and investigating the influence of key hyperparameters are presented in Section \ref{NumExp}. Finally, a brief summary and discussion are given in Section \ref{Summary}.

\section{Preliminaries}\label{Preliminaries}
\subsection{Weak Solution and Interior Regularity}
In this subsection, we mainly provide some existing regularization theories for the weak solution of divergence form elliptic PDEs. Consider a general boundary value problem

\begin{equation}
  \left\{\begin{aligned}
  &L u=f  \text { in } \Omega, \\
  &u=0  \text { on } \partial \Omega,
  \end{aligned}\right.
  \label{Equation:1} 
\end{equation}
where $\Omega$ is an open, bounded subset of $\mathbb{R}^n, u=u(x)$ is defined on $\Omega, f\in L^2(\Omega)$, and $L$ denotes a second-order partial differential operator having the divergence form
\begin{equation}
  L u=-\sum_{i, j=1}^n\left(a^{i j}(x) u_{x_i}\right)_{x_j}+\sum_{i=1}^n b^i(x) u_{x_i}+c(x) u
  \label{Equation:2}
\end{equation}
for given coefficient functions $a^{i j}, b^i, c(i, j=1, \ldots, n)$. 

The partial differential operator $L$ is uniformly elliptic if there exists a constant $\mu>0$ such that
$$
\sum_{i, j=1}^n a^{i j}(x) \xi_i \xi_j \geq \mu|\xi|^2
$$
holds, for a.e. $x \in \Omega$ and all $\xi \in \mathbb{R}^n$. Uniformly ellipticity implies that, for each point $x \in \Omega$, the symmetric $n \times n$ matrix function $\mathbf{A}(x)=\left(a^{i j}(x)\right)$ is positive definite, with the smallest eigenvalue bounded below by $\mu$.

Since classical solutions may not exist, the concept of weak solutions is considered instead. The weak solution of Equation (\ref{Equation:1}) is defined as follows.

\begin{definition}
(Weak Solution) The bilinear form $B[\cdot,\cdot]$ associated with the divergence form elliptic operator $L$ defined by (\ref{Equation:2}) is
  \begin{equation}
    B[u, v] := \int_{\Omega} \left( \sum_{i,j=1}^n a^{ij} u_{x_i} v_{x_j} + \sum_{i=1}^n b^i u_{x_i} v + c u v \right) dx
  \end{equation}
  for $u, v \in H_0^1(\Omega)$. Then $u \in H_0^1(\Omega)$ is a weak solution of the boundary-value problem (\ref{Equation:1}) if
  \begin{equation}
    B[u, v] = (f, v)
  \end{equation}
  for all $v \in H_0^1(\Omega)$, where $(\cdot,\cdot)$ denotes the inner product in $L^2(\Omega)$.
\end{definition}
 
Suppose $u \in H_0^1(\Omega)$ is a weak solution of the Equations (\ref{Equation:1}), where $L$ has the divergence form. Then the following result holds.

\begin{theorem}
    (Interior regularity \cite{evans2022partial}) Let $k$ be a nonnegative integer, and assume
      $$
      a^{i j}, b^i, c \in C^{k+1}(\Omega) \quad(i, j=1, \ldots, n)
      $$
      and
      $$
      f \in H^k(\Omega) .
      $$   
    Suppose furthermore that $u \in H^1(\Omega)$ is a weak solution of the elliptic PDE
    $$
    L u=f \quad \text { in } \Omega \text {. }
    $$
    Then
    $$
    u \in H_{\text {loc }}^{k+2}(\Omega) ;
    $$
    and for each open subset $\Omega^{\prime} \subseteq \Omega$ the estimate 
    $$
    \|u\|_{H^{k+2}(\Omega^{\prime})} \leq C\left(\|f\|_{H^k(\Omega)}+\|u\|_{L^2(\Omega)}\right)
    $$
    holds, where the constant $C$ depending only on $k, \Omega^{\prime}, \Omega$, and the coefficients of $L$.
    \label{Theorem:1}
\end{theorem}

\begin{theorem}
    If $f \in H^k(\Omega) $ for $ k>\frac{n}{2}$, and $u$ is the weak solution of Equation (\ref{Equation:1}). Then for each open subset $\Omega^{\prime} \subseteq \Omega$, we claim that $u\in C^{2,\alpha}(\Omega^{\prime})$, where $0<\alpha \leqslant 1-\frac{n}{2 k}$.
    \label{lemma1}
\end{theorem}
\begin{proof}
  According to Theorem \ref{Theorem:1}, for any open subset $\Omega^{\prime} \subseteq \Omega, u \in H^{k+2}(\Omega^{\prime})$. Based on Sobolev Embedding Theorem \cite{evans2022partial},
  when $k>\frac{n}{2}, H^{k+2}(\Omega^{\prime}) \hookrightarrow C^{2,\alpha}(\Omega^{\prime})$, where $0<\alpha \leqslant 1-\frac{n}{2 k}.$
\end{proof}

When $a^{i j} \equiv \delta_{i j}, b^i \equiv 0, c \equiv 0$, the operator $L = -\Delta$. Then equation (\ref{Equation:1}) degenerates to the Poisson equation
\begin{equation}\nonumber
  \left\{\begin{aligned}
  &- \Delta u=f  \text { in } \Omega, \\
  &u=0  \text { on } \partial \Omega.
  \end{aligned}\right.
  \label{Equation:4}
\end{equation}
When $f \in C^{\infty}$, it follows that the solution $u\in C^{\alpha}$ for some $\alpha \in (0,1)$. This prior regularity implies that the neural network-based approximate solution is also expected to exhibit H\"older continuity.

\subsection{Property of H\"older Seminorm}
In this subsection, following the framework in \cite{krylov1996lectures}, we define the global H\"older seminorm and the local H\"older seminorm of the network mapping $u_\theta(x): \Omega \rightarrow \mathbb{R}^m$ as follows.

\begin{definition}%\cite{krylov1996lectures}
  Let $u_\theta(x):\Omega \to \mathbb{R}^m$ be a neural network mapping parameterized by
  % function on $\Omega \subset \mathbb{R}^n
  $\theta, 0<\alpha <1, u_\theta(x)$ is H\"older continuous with exponent $\alpha$, if 
$$|u_\theta(x)-u_\theta(y)| \leq C|x-y|^\alpha, \quad x, y \in \Omega,$$
for a constant $C$. Then the global H\"older seminorm of $u_\theta(x)$ is defined as
$$
[u_\theta]_{\alpha ; \Omega}=\sup _{x, y \in \Omega, x \neq y} \frac{|u_\theta(x)-u_\theta(y)|}{|x-y|^\alpha}.
$$
The global H\"older norm is
$$
\Vert u_\theta \Vert_{\alpha ; \Omega} = \sup_{x \in \Omega} |u_\theta(x)| + \sup _{x, y \in \Omega, x \neq y} \frac{|u_\theta(x)-u_\theta(y)|}{|x-y|^\alpha}.
$$
% the H\"older norm of $u$.
Here, $|\cdot |$ represents the Euclidean norm of a vector.
\label{Definition:1}
\end{definition}
However, the computation of the global H\"older norm is relatively complex. To address this, the local H\"older seminorm is introduced to replace it.

\begin{definition}%\cite{krylov1996lectures}
  $u_\theta(x)$ is local H\"older continuous with exponent $\alpha$ at point $x$, if
$$
|u_\theta(x)-u_\theta(y)| \leq C|x-y|^\alpha, \quad y \in \Omega_\rho(x),
$$
for a constant $C$, where $\Omega_\rho(x)=\Omega \cap B_\rho(x)$, and $B_\rho(x)$ is a ball of radius $\rho$ that does not contain the center $x$. The local H\"older seminorm of $u_\theta(x)$ at $x$ is given by
$$
[u_\theta]_{\alpha ; \Omega}^{l o c}(x)=\int_{\Omega_\rho(x)} \frac{|u_\theta(x)-u_\theta(y)|}{|x-y|^\alpha} \rm{d} y .
$$
\label{Definition:2}
\end{definition}

Nevertheless, to further ease the calculation and facilitate subsequent work, we provide a sufficient condition for local H\"older continuity.
\begin{lemma}
  Assume that $u_\theta(x):\Omega \mapsto \mathbb{R}^m$ is bounded. If
  $$\sup_{y \in \Omega_\rho(x)} \frac{|u_\theta(x)-u_\theta(y)|}{|x-y|^\alpha}<+\infty,$$
  then the local H\"older seminorm of function $u$ at point $x$ 
  $$\int_{\Omega_\rho(x)} \frac{|u_\theta(x)-u_\theta(y)|}{|x-y|^\alpha} \rm{d} y$$
  exists.
\label{lemma2}
\end{lemma}
\begin{proof}
Since the supremum
$$
\sup_{y \in \Omega_\rho(x)} \frac{|u_\theta(x) - u_\theta(y)|}{|x - y|^\alpha} = M < +\infty,
$$
there exists a constant $M > 0$ such that for all $y \in \Omega_\rho(x)$,
$$
\frac{|u_\theta(x) - u_\theta(y)|}{|x - y|^\alpha} \leq M.
$$
Consequently, the integral can be bounded by
$$
\int_{\Omega_\rho(x)} \frac{|u_\theta(x) - u_\theta(y)|}{|x - y|^\alpha} \, \mathrm{d}y \leq \int_{\Omega_\rho(x)} M \, \mathrm{d}y = M \, |\Omega_\rho(x)|,
$$
where $|\Omega_\rho(x)|$ denotes the measure of the set $\Omega_\rho(x)$.

Since $|\Omega_\rho(x)|$ is finite (as $\Omega_\rho(x)$ is a bounded neighborhood), the integral is finite, and thus the local H\"older seminorm exists.
\end{proof}

\section{Methodology}\label{Methodology}%PINNs Based on Local H\"older Norm}
This section outlines the core methodology and divided into four parts. The first part provides an introduction to PINNs. The second part formulates a novel loss function motivated by the regularity theory of elliptic PDEs. The third part presents a discrete approximation scheme for practical computation. The fourth part designs a numerical algorithm for solving elliptic PDEs and provides a theoretical analysis of the approximation error.

In this paper, we focus on the Dirichlet boundary value problem for second-order elliptic partial differential equations, including both divergence and non-divergence forms. Specifically, we study PDEs of the form
\begin{equation}
  \left\{\begin{aligned}
  &L u(x)=f(x),\quad x \in \Omega, \\
  &u(x)=g(x),\quad x \in \partial \Omega.
  \end{aligned}\right.
  \label{Equation:9} 
\end{equation}
Here, $f\in L^2(\Omega), g\in H^1(\Omega), u(x)$ denotes the solution defined on the spatial domain $\Omega$ with variable $x$, and $L$ is a second-order elliptic operator. 

\subsection{PINNs}\label{PINNMETHOD}
PINNs solve PDEs by incorporating physical laws directly into the training process. In this framework, the solution is approximated by a deep neural network $u_\theta$, which maps spatial coordinates to the predicted solution, $\theta$ denotes the trainable parameters. The differential operator $L$ is computed via automatic differentiation.

To measure the discrepancy between the neural network prediction $u_\theta$ and the PDEs constraints, we define the loss function as a weighted sum of the $L^2$ norms of the residuals from the governing equation and boundary conditions:
\begin{equation}\label{Equation:5}
  \begin{aligned}
    \mathcal{L}_{PINNs}(\theta)&=w_r \mathcal{L}_r\left(\theta \right)+w_b \mathcal{L}_b\left(\theta \right),\\
    \mathcal{L}_r\left(\theta \right)&=\Vert Lu_\theta(x)-f(x)\Vert_{2,\Omega}^2,\\
    \mathcal{L}_b\left(\theta \right)&=\Vert u_\theta(x)-g(x)\Vert_{2,\partial\Omega}^2.
    \end{aligned}
\end{equation}
In equations (\ref{Equation:5}), $w_r$ and $w_b$ are the weights for the residual loss $\mathcal{L}_r\left(\theta \right)$ and boundary loss $\mathcal{L}_b\left(\theta \right)$. 
Since evaluating the loss function $\mathcal{L}_{PINNs}(\theta)$ over the entire domain $\Omega$ is impractical, it is instead approximated by the mean squared error of the training dataset.
The training dataset $X$ consists of an unsupervised dataset $X_r$ and a supervised dataset $X_b$
$$
X_r = \left\{ x_r^i | x_r^i \in \Omega\right\}_{i=1}^{N_r},\quad
X_b = \left\{ \left(x_b^i, g(x_b^i)\right) | x_b^i \in \partial \Omega\right\}_{i=1}^{N_b}.
$$
Therefore, the loss function $\mathcal{L}_{PINNs}(\theta)$ can be approximated by
\begin{equation}
  \begin{aligned}
    \mathcal{L}_{PINNs}(\theta;X)&=w_r \mathcal{L}_r\left(\theta; X_r\right)+w_b \mathcal{L}_b\left(\theta; X_b\right),\\
    \mathcal{L}_r\left(\theta ; X_r\right)&=\dfrac{1}{N_r}\sum_{i=1}^{N_r} \left|Lu_\theta(x_{r}^{i})-f(x_{r}^{i})\right|^2,\\
    \mathcal{L}_b\left(\theta ; X_b\right)&=\dfrac{1}{N_b}\sum_{i=1}^{N_r} \left|u_\theta(x_{b}^{i})-g(x_{b}^{i})\right|^2.
  \end{aligned}
  \label{Equation:6}
\end{equation}

Therefore, the problem of solving the partial differential equation is equivalent to the following parameter optimization problem
\begin{equation}\nonumber
  \theta^*=\arg \min _\theta \mathcal{L}_{PINNs}(\theta ; X) .
  \label{Equation:8}
\end{equation}
To obtain the optimal network parameters $\theta^*$, various optimization algorithms can be employed, such as gradient descent, stochastic gradient descent (SGD), Adam, and L-BFGS.

The universal approximation theorem \cite{cybenko1989approximation} forms the theoretical basis for using deep neural networks (DNNs) to solve PDEs, as it ensures their expressive power for representing complex solutions.
Nevertheless, practical challenges persist. These include the design of suitable network architectures and hyperparameters, the need for adequate training data, and the high computational cost of optimization.

\subsection{Loss Function Guided by Local H\"older Regularity}
Motivated by the interior regularity theory of elliptic PDEs, we introduce a novel regularization strategy within the PINNs framework. Specifically, we incorporate the local H\"older seminorm of the solution $u(x)$, the solution to Equation (\ref{Equation:9}), into the loss function. This regularization term promotes local smoothness in the learned solution and enhances generalization. We define the following regularized loss function

\begin{equation}
  \mathcal{L}(\theta )=\mathcal{L}_{PINNs}(\theta )+\lambda \mathcal{L}_{local-H\ddot{o}lder}(\theta ),
  \label{Equation:10}
\end{equation}
where $\mathcal{L}_{PINNs}(\theta )$ is defined in  (\ref{Equation:5}), and  $\mathcal{L}_{local-H\ddot{o}lder}(\theta )$ is defined as
\begin{equation}
  \mathcal{L}_{local-H\ddot{o}lder}(\theta ) = \sup_{x \in \Omega} |u_\theta(x)| + \sup_{x \in \Omega} \sup _{y \in \Omega_\rho(x)} \frac{|u_\theta(x)-u_\theta(y)|}{|x-y|^\alpha}.
  % \end{aligned}
  \label{Equation:11}
\end{equation}
Here, $\Omega_\rho(x)$ represents a neighborhood centered at $x$ with radius $\rho$.

The proposed regularized loss function in Equation~(\ref{Equation:10}) includes an additional local H\"older seminorm term that enforces smoothness of the neural network output within local neighborhoods. When the regularization coefficient $\lambda$ is set to zero, the loss reduces to the standard PINNs formulation. This H\"older regularization term can be interpreted as embedding a local smoothness prior into the training process. 

In practice, evaluating the local H\"older seminorm over the entire continuous domain $\Omega$ is computationally infeasible, as it requires computing the supremum over all point pairs within each neighborhood $\Omega_\rho(x)$. 
To address this, we discretize the domain $\Omega$ into a finite set of residual points $X_r \subset \Omega$, which are commonly used in the PINNs framework to enforce the differential equation residual. The local H\"older seminorm is then approximated by restricting the evaluation to these residual points. Specifically, we approximate Equation~(\ref{Equation:11}) as

\begin{equation}\nonumber
  \mathcal{L}_{local-H\ddot{o}lder}(\theta ; X_r) = \sup_{x \in X_r} |u_\theta(x)| + \sup_{x \in X_r} \sup _{y \in \Omega_\rho(x)} \frac{|u_\theta(x)-u_\theta(y)|}{|x-y|^\alpha}.
  \label{Equationlisan}
\end{equation}

This discrete formulation enables efficient numerical implementation while preserving the essence of local H\"older continuity. It ensures that the neural network solution maintains local regularity over the representative training points in $X_r$, which are typically dense enough to capture the behavior of the solution across the domain.

\subsection{Calculation Scheme Based on Variable-Distance Sampling}\label{sec:variable}

According to Definition~\ref{Definition:2}, the local H\"older seminorm of the function $u_\theta$ at a point $x$ is defined as an integral over the neighborhood $\Omega_\rho(x)$. However, directly evaluating this quantity is computationally intractable, as it requires computing a singular integral over a continuous domain.
To overcome this challenge, we adopt a discrete approximation of the seminorm, guided by Lemma~\ref{lemma2}, which provides a sufficient condition for the existence of the local H\"older seminorm. Specifically, we discretize the solution domain $\Omega$ into a finite set of residual points $X_r = \{x_r^i \in \Omega\}_{i=1}^{N_r}$, and approximate each neighborhood $\Omega_\rho(x)$ using a finite collection of sample points. This approach enables efficient and scalable computation of the H\"older term while maintaining the desired local regularity constraints.

Before introducing the discrete formulation, we present a regularity assumption that quantifies the approximation error caused by discretization.
\begin{assumption}
Let \( u_\theta(x): \Omega \rightarrow \mathbb{R}^m \). When the neighborhood \( \Omega_\rho(x) \) is approximated by a discrete set of points  $\tilde{\Omega}_\rho(x) = \left\{ y_i|y_i\in \Omega_\rho(x) \right\}_{i=1}^{N_H}$, we assume the discretization error of the local H\"older seminorm satisfies
\begin{equation}\nonumber
% \begin{aligned}
\left|\sup _{y \in \Omega_\rho(x)} \frac{|u_\theta(x)-u_\theta(y)|}{|x-y|^\alpha} - \sup _{y \in \tilde{\Omega}_\rho(x)} \frac{|u_\theta(x)-u_\theta(y)|}{|x-y|^\alpha} \right| \leq C_{sup} N_H^{-\alpha_{sup}},
% \end{aligned}
\end{equation}
where \( \alpha_{\text{sup}} > 0 \), and the constant \( C_{\text{sup}} \) depends on the domain dimension and the radius \( \rho \). Moreover, as \( \rho \to 0 \), \( C_{\text{sup}} \to 0 \).
\end{assumption}

In the variable-distance scheme, the local neighborhood $\Omega_\rho(x)$ is approximated by a set of randomly distributed points within a closed ball of radius $\rho$ centered at $x$. Formally, the discrete neighborhood is defined as
\[
\Omega_\rho(x) \approx \tilde{\Omega}_\rho(x) = \{x\} \oplus \kappa_{\rho}, \quad\kappa_\rho =  \{y_i \mid y_i = \rho v_i,\ v_i \in \mathbb{B}^n(0,1)\}_{i=1}^{N_H},
\]
where $v_i$ are randomly sampled vectors inside the unit ball $\mathbb{B}^n(0,1)$ in $\mathbb{R}^n$, and $N_H$ is a hyperparameter. In the left part of Figure~\ref{fig:1}, the lower image within the training set visualization depicts the distribution of $\kappa_{\rho}$ in the two-dimensional case.

Based on this discretization, the local H\"older loss is approximated by
\begin{equation}
\mathcal{L}_{local-H\ddot{o}lder}^{Variable}(\theta; X_r) = \sup_{x \in X_r} |u_\theta(x)| + \sup_{x \in X_r} \sup_{y \in \kappa_\rho} \frac{|u_\theta(x) - u_\theta(x+y)|}{|y|^\alpha},
\label{varholder}
\end{equation}
and the total loss function is
\begin{equation}\label{total}
\mathcal{L}(\theta; X) = \mathcal{L}_{PINNs}(\theta; X) + \lambda \mathcal{L}_{local-H\ddot{o}lder}^{Variable}(\theta; X_r).
\end{equation}

Based on the variable-distance sampling strategy described above, the complete training procedure for the proposed method is presented in Algorithm~\ref{alm1}, and an illustration of the training process is provided in Figure~\ref{fig:1}.
\begin{figure*}
  \centering
  \includegraphics[width=\textwidth]{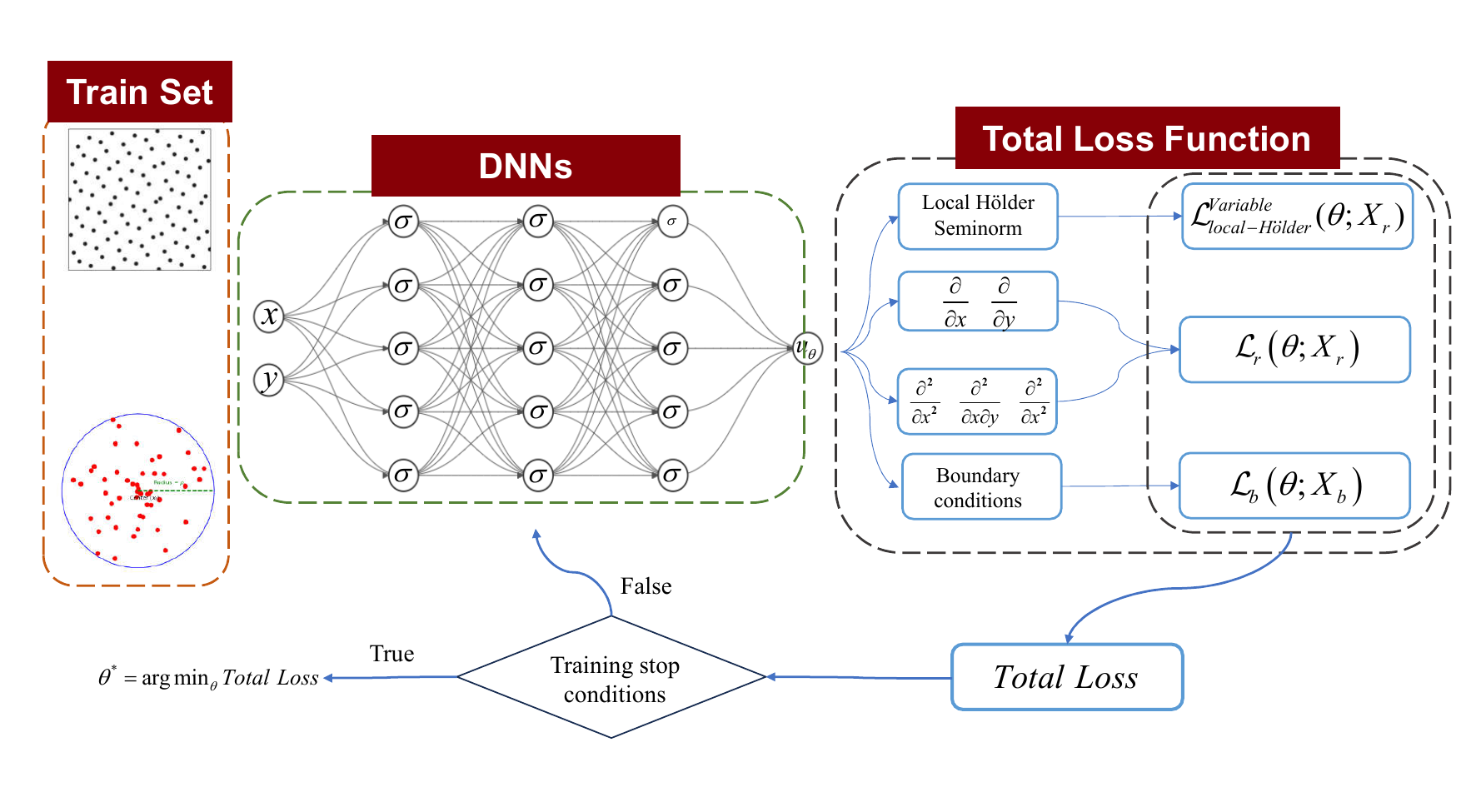}
  \caption{Network Architecture: The training data includes residual points (upper left figure) and the point set used to calculate the Local H\"older regularization term (lower left figure)}
  \label{fig:1}       % Give a unique label
\end{figure*}

\begin{algorithm}[htbp]
  \caption{PINNs with Local H\"older Regularization (Variable-Distance Sampling)}
  \label{alm1}
  \begin{algorithmic}
  \State \textbf{Input}  Network size, The training dataset $X_r, \kappa_{\rho}$, the regularization parameters $\lambda$ and H\"older indices $\alpha$.
  \State \textbf{Step 1} Define a neural network $u_\theta(x)$ and initialize parameters $\theta_0$.
  \State \textbf{Step 2} Calculate the residual loss $\mathcal{L}_{r}(\theta ; X_r)$ according to Equation (\ref{Equation:6}).
  \State \textbf{Step 3} Calculate regularization loss term using Equation (\ref{varholder}).
  \State \textbf{Step 4} Calculate total loss according to Equation(\ref{total})
  \State \textbf{Step 5} Train the neural network to find the best parameters $\theta^*$ by minimizing the total loss function;
  \State \textbf{End}
  \end{algorithmic}
  \end{algorithm}

\subsection{Error Estimation} 
In this subsection, we derive the generalization error estimate of the proposed method. To facilitate the analysis, we first introduce two fundamental assumptions that are essential for the proof.
\begin{assumption}\cite{mishra2023estimates}\label{stable}
Let \( X \) and \( Y \) be Banach spaces equipped with norms \( \| \cdot \|_X \) and \( \| \cdot \|_Y \), respectively. Consider an elliptic partial differential operator \( L: X \to Y \). For any \( u, v \in X \), the following stability estimate holds:
\[
\| u - v \|_X \leq C_{pde} \|Lu - Lv\|_Y,
\]
where the constant \( C_{pde} > 0 \) depends explicitly on \( \| u \|_X \) and \( \| v \|_X \).
\end{assumption}

\begin{remark}
Assumption~\ref{stable} represents a standard stability condition for PDEs. In the linear case, if the operator \( L: X \to Y \) admits a bounded inverse \( L^{-1}: Y \to X \), then the assumption reduces to the boundedness of \( L^{-1} \), i.e., \( \|L^{-1}\| \leq C < \infty \). This holds for many well-posed linear PDEs.
\end{remark}

To rigorously analyze the approximation error, it is necessary to consider the numerical integration error. We therefore introduce the following assumption to characterize this integration error.
\begin{assumption}\label{quadr}
Consider a function $g:\Omega \to \mathbb{R}^m,$ 
and the integral 
\[
I(g) := \int_{\Omega} g(x) \, dx.
\]
To approximate this integral, we employ a quadrature rule with \( N \) quadrature points \( x_i \in \Omega \) and positive weights \( w_i > 0 \), yielding the numerical approximation
\[
I_N(g) := \sum_{i=1}^{N} w_i g(x_i).
\]
We assume that the quadrature rule satisfies the following error estimate
\begin{equation}
|I(g) - I_N(g)| \leq C_{quad} N^{-\alpha}, \quad \alpha > 0,
\end{equation}
where the constant \( C_{quad} \) depends on the dimension of the domain and the property of the integrand \( g \).

\end{assumption}

Based on the two assumptions presented above, we are now in a position to establish a generalization error bound for the proposed method.

\begin{theorem}\label{theorem5}
  Let $u^*(x)\in C^{2,\alpha}(\Omega)$ be the unique solution of Equation(\ref{Equation:9}), and $u_{\theta}(x)$ is the approximate solution obtained via the proposed method. Then, under Assumptions \ref{stable} and \ref{quadr}, the generalization error is bounded by
 \begin{equation}
\| u_\theta(x) - u^*(x) \|_{2,\alpha;\Omega} \leq C_{pde}\left(\mathcal{L}(\theta ; X) + C_{quad}N_r^{-\alpha_r}\right),
  \end{equation}
where $\mathcal{L}(\theta ; X)$ are defined in (\ref{total}), and $N_r$ represents the number of internal sampling points.
\end{theorem}

\begin{proof}
  By setting $u = u_{\theta}(x),v=u^*(x)$ in Assumption \ref{stable}, we obtain
  \begin{equation}
    \begin{aligned}
    \| u_\theta(x) - u^*(x) \|_{2,\Omega}^2 &\leq C_{pde} \|Lu_\theta(x) - Lu^*(x)\|_{2,\Omega}^2\\
    & \leq C_{pde} \|Lu_\theta(x) - f(x)\|_{2,\Omega}^2\\
    & = C_{pde} \mathcal{L}_{r}(\theta) .
    \end{aligned} 
  \end{equation}
  According Assumption \ref{quadr}, we obtain
  \begin{equation}
    \mathcal{L}_{r}(\theta) = \|Lu_\theta(x) - Lu^*(x)\|_{2,\Omega}^2 \leq \mathcal{L}_r\left(\theta ; X_r\right) + C_{quad}N_r^{-\alpha_r}
  \end{equation}
  Since the residual loss $\mathcal{L}_r(\theta; X_r)$ is a component of the total loss $\mathcal{L}(\theta ; X)$ , we conclude that:
  \[\| u_\theta(x) - u^*(x) \|_{2,\alpha;\Omega} \leq C_{pde}\left(\mathcal{L}(\theta ; X) + C_{quad}N_r^{-\alpha_r}\right).\]
  Finally, by combining $u^*(x)\in C^{2,\alpha}(\Omega)$ and the error estimates, we obtain 
  \[
  \begin{aligned}
      \| u_\theta(x)\|_{2,\alpha;\Omega} &\leq \| u_\theta(x) - u^*(x) \|_{2,\alpha;\Omega} + \| u^*(x) \|_{2,\alpha;\Omega} \\
      &\leq C_{pde}\left(\mathcal{L}(\theta ; X) + C_{quad}N_r^{-\alpha_r}+ \| u^*(x) \|_{2,\alpha;\Omega}\right),
  \end{aligned}\]
  which means $u_\theta(x)\in C^{2,\alpha}(\Omega)$. This completes the proof.
\end{proof}

\section{Numerical Experiment}\label{NumExp}
In this section, the PINNs based on the local H\"older Regularity, proposed in Section~\ref{Methodology}, is applied to three distinct types of elliptic PDEs. For each type of numerical experiment, the influence of various hyperparameter combinations on the model performance is systematically investigated. The evaluation metrics include the Relative Mean Square Error (RMSE) and the $L^{\infty}$ norm

$$
\begin{aligned}
  RMSE (u, u_\theta) &= \dfrac{\sqrt{\sum_{i=1}^{N}|u_\theta(x_i)-u(x_i)|^2}}{\sqrt{\sum_{i=1}^{N}|u(x_i)|^2}},\\
  \Vert u-u_\theta \Vert_{\infty} &= \max_{1 \leq i \leq N} |u_\theta(x_i)-u(x_i)|,
\end{aligned}
$$
where $N$ represent the number of test points. To provide a baseline for comparison, the standard PINNs method described in Section~\ref{PINNMETHOD} is also implemented. All training is conducted on a system configured with an NVIDIA GeForce RTX 2060 GPU (1920 CUDA cores, 6GB GDDR6 memory), an Intel Core i7-11700 processor, and 32GB of RAM.

For boundary conditions, a penalty-free method is modifying the network architecture\cite{lu2021physics}. 
The approximate solution of the PDEs is constructed by
\begin{equation}\nonumber
  \hat{u}_{\theta}(x,y)=u_0(x,y)+\ell(x) \mathcal{N}_{\theta}(x,y),
\end{equation}
where $u_0(x)$ is the boundary function, $\mathcal{N}_{\theta}(x)$ is the outputs of the DNNs which parametric by $\theta$, $\ell(x)$ is a function satisfying the following two conditions
\begin{equation}\nonumber
  \begin{cases}\ell(x,y)=0, & (x,y) \in \partial \Omega, \\ \ell(x,y)>0, & (x,y) \in \Omega .\end{cases}
\end{equation}
This approach satisfies the boundary conditions exactly; therefore, no sampling on the boundary is required, which in turn reduces the overall computational cost.

\subsection{Second-Order ODE}
We consider a second-order ordinary differential equation with the boundary condition
\begin{equation}
  \left\{\begin{aligned}
  &-\frac{d^2 u(x)}{d x^2}=f(x) \quad  x\in(-\pi,\pi), \\
  &u(-\pi) = u(\pi) = 0,
\end{aligned}\right.
  \label{one-Dimension}
\end{equation}
where $f(x)=\sum_{i=0}^{3}(2i+1)\sin((2i+1)x)$. The corresponding exact solution has the analytical form $u(x)=\sum_{i=0}^{3}\sin((2i+1)x)/(2i+1)$.
The exact solution is shown in Figure \ref{simple_gt}.

To solve the problem, neural networks with architecture \([1, 20, 20, 20, 1]\) are constructed. The \(\tanh(x)\) is chosen as the activation function. The parameters are set as \(N_r = 100, w_r = 1\) and \(\lambda = 1e-3\). For testing and visualization purposes, a uniform mesh grid with a step size of \(0.01\pi\) is generated.

Figure \ref{fig:simple} shows the experimental results obtained using Standard PINNs and the two types of models proposed in this paper.
As shown in Figure \ref{simple_standard_pinn_pre} and \ref{simple_standard_pinn_abs_Error}, the prediction errors of the Standard PINNs are concentrated in the region of (-2, 0), with the maximum absolute error being approximately 0.020. 
However, in the model proposed in this paper, 20 additional points are introduced in the vicinity of each residual point. The resulting predictions are illustrated in Figure~\ref{simple_var_pre} and the corresponding absolute errors are shown in Figure~\ref{simple_var_abs_Error}. The maximum absolute error achieved by the proposed method is approximately 0.008, representing a reduction of about 60\% compared to the standard PINNs. A quantitative comparison of the performance of all three methods on the test set is provided in Table~\ref{tab:simple}.

The resulting improvement is significant, which further demonstrates the effectiveness of the models we propose. In this simple example, we did not deliberately fine-tune the parameters to obtain the optimal results for each method. We merely added a regularization term to the loss function of the standard PINNs. 

\begin{table*}[htbp]
  \centering
  \caption{Testing Error of Second-Order ODE($\alpha = 1/2, \rho = 0.01$)}
  \label{tab:simple}       
  \tabcolsep=0.3cm
  \renewcommand\arraystretch{1.25}
  \begin{tabular}{c|cc}
    \hline
    Method & Standard PINN &  Ours \\
    \hline
    RMSE & 0.01765 &  0.00494 \\
    Bounds of Absolute Error & 2.09782e-02 &  7.79681e-03 \\
    \hline
    \end{tabular}
  \end{table*}

\begin{figure*}[htbp]
  \centering
  \begin{subfigure}[b]{0.4\textwidth}
    \centering
    \includegraphics[width=\textwidth]{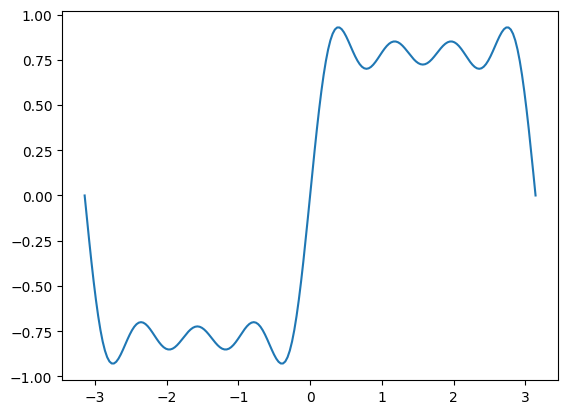}
    \caption{Ground Truth}
    \label{simple_gt}
  \end{subfigure}\\
  \begin{subfigure}[b]{0.4\textwidth}
    \centering
    \includegraphics[width=\textwidth]{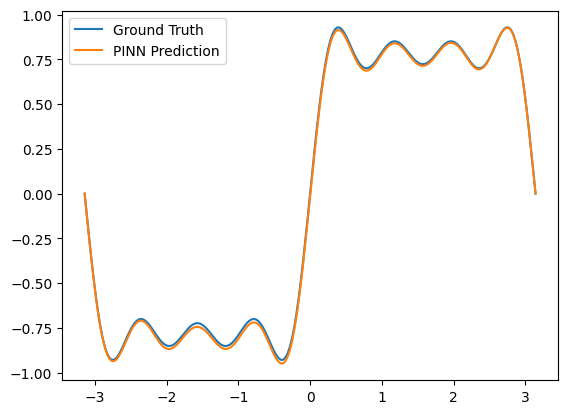}
    \caption{\centering{Standard PINNs Prediction}}
    \label{simple_standard_pinn_pre}
  \end{subfigure}
  \begin{subfigure}[b]{0.4\textwidth}
    \centering
    \includegraphics[width=\textwidth]{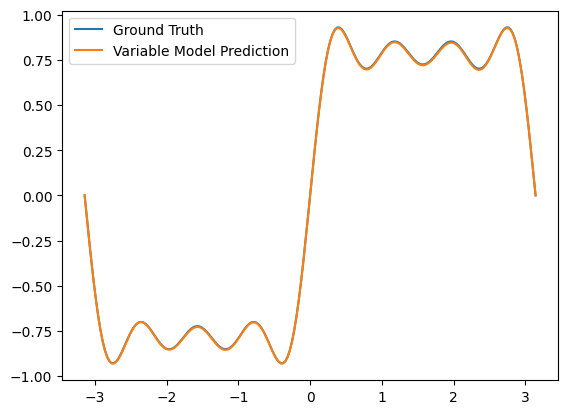}
    \caption{\centering{Ours Prediction($N_H=20$)}}
    \label{simple_var_pre}
  \end{subfigure}
  
  \begin{subfigure}[b]{0.4\textwidth}
    \centering
    \includegraphics[width=\textwidth]{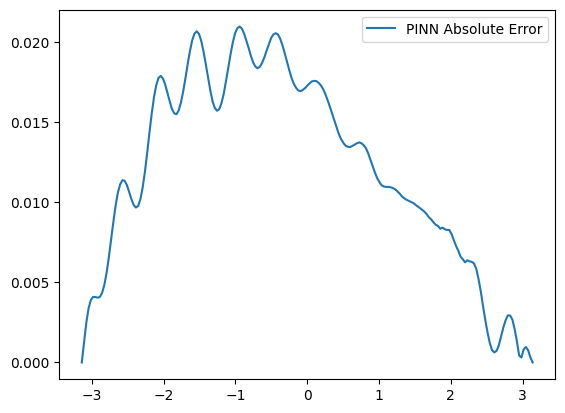}
    \caption{\centering{Standard PINNs Absolute Error}}
    \label{simple_standard_pinn_abs_Error}
  \end{subfigure}
  \begin{subfigure}[b]{0.4\textwidth}
    \centering
    \includegraphics[width=\textwidth]{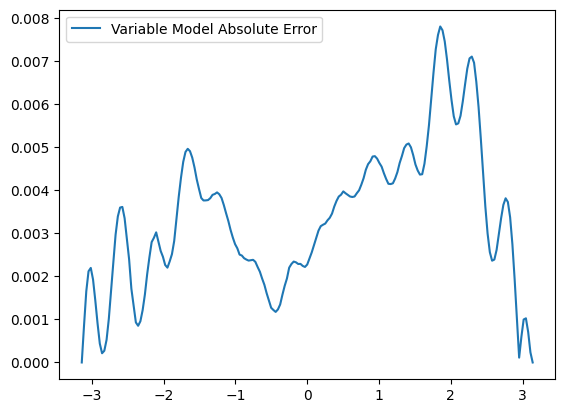}
    \caption{\centering{Ours Absolute Error($N_H=20$)}}
    \label{simple_var_abs_Error}
  \end{subfigure}

  \caption{The result of Second-Order ODE: $\alpha = 1/2, \rho = 0.01.$}
  \label{fig:simple}
\end{figure*}

\subsection{Poisson Equation}
In this subsection, we consider the Dirichlet boundary value of the Poisson equation as follows
\begin{equation}
  \left\{\begin{aligned}
  &- \Delta u(x, y) = f(x, y), \quad  (x, y) \in \Omega, \\
  &u(x, y) = g(x, y), \quad (x, y) \in \partial \Omega,
  \end{aligned}\right.
  \label{Poisson Equation}
\end{equation}
where $\Omega = (-1,1)\times (-1,1)$. The exact solution is specified as
\[
u(x, y) = \left(0.1 \sin (2 \pi x) + \tanh (10 x)\right) \sin (2 \pi y).
\]
The corresponding source term $f(x, y)$ and boundary condition $g(x, y)$ are derived analytically to ensure that the exact solution of Equation~(\ref{Poisson Equation}) is given by $u(x, y)$. The ground truth is visualized in Figure~\ref{Poisson_gt}.

\begin{figure*}[htbp]
  \centering
  \begin{subfigure}[b]{0.35\textwidth}
    \centering
    \includegraphics[width=\textwidth]{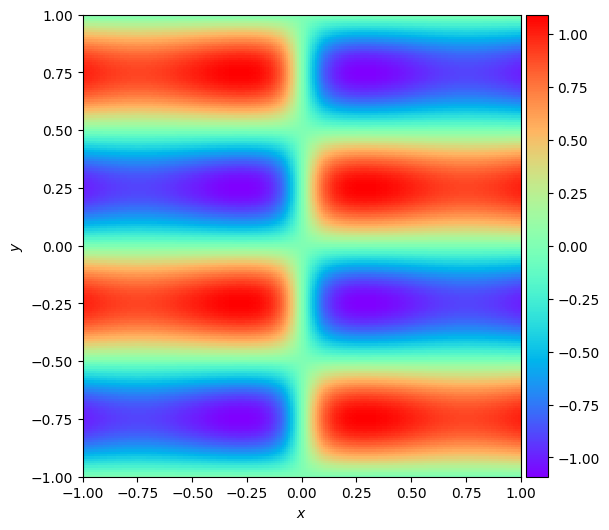}
    \caption{Ground Truth}
    \label{Poisson_gt}
  \end{subfigure}\\
  \begin{subfigure}[b]{0.35\textwidth}
    \centering
    \includegraphics[width=\textwidth]{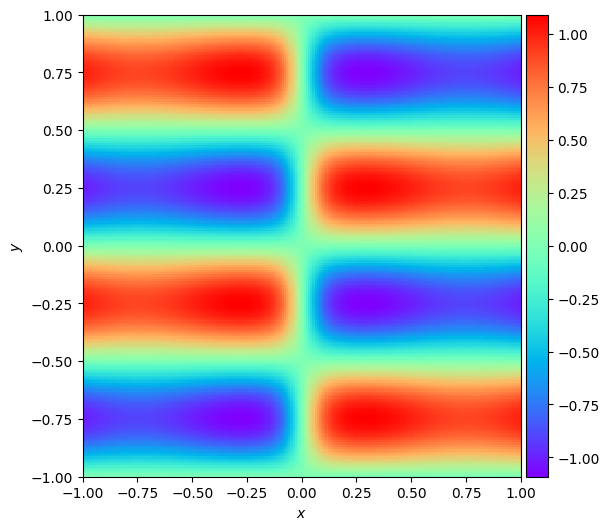}
    \caption{Standard PINNs Prediction}
    \label{poisson_standard_pinn_pre}
  \end{subfigure}
  \begin{subfigure}[b]{0.35\textwidth}
    \centering
    \includegraphics[width=\textwidth]{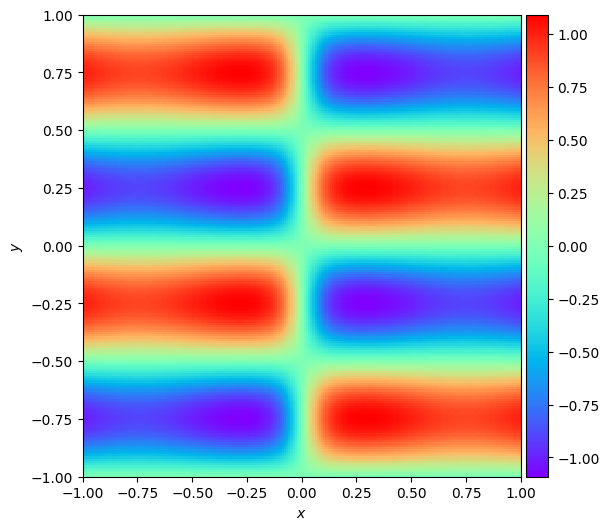}
    \caption{Ours Prediction}
    \label{poisson_var_pre}
  \end{subfigure}

  \begin{subfigure}[b]{0.35\textwidth}
    \centering
    \includegraphics[width=\textwidth]{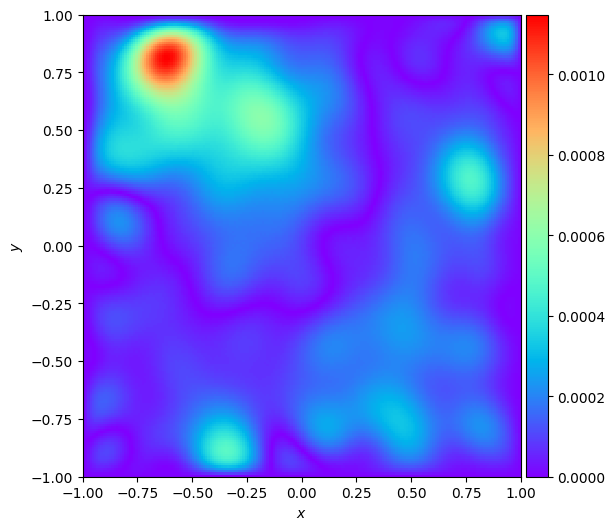}
    \caption{\centering{Standard PINNs \protect\\Absolute Error}}
    \label{poisson_standard_pinn_abs_Error}
  \end{subfigure}
  \begin{subfigure}[b]{0.35\textwidth}
    \centering
    \includegraphics[width=\textwidth]{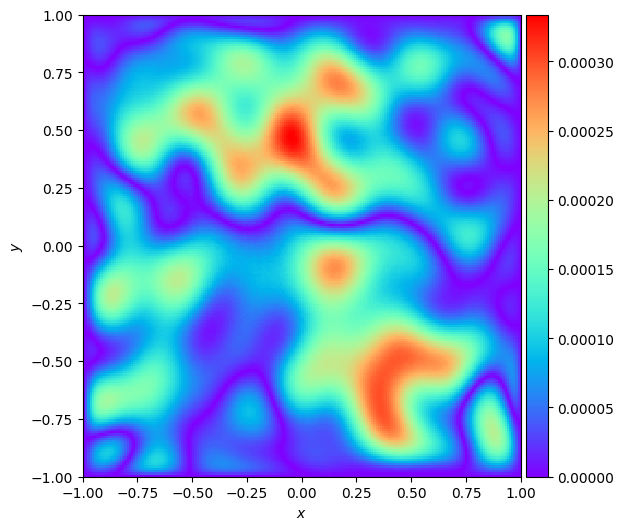}
    \caption{\centering{Ours \protect\\Absolute Error}}
    \label{poisson_var_abs_error}
  \end{subfigure}
  \caption{The result of Poisson Equation: $N_H=15, \alpha = 1/2, \rho = 0.005.$}
  \label{fig:test1}
\end{figure*}

To solve the problem, neural networks with architecture \([2, 20, 20, 20, 20, 1]\) are constructed. The \(\tanh(x)\) is chosen as the activation function. The parameters are set as \(N_r = 400, w_r = 1\) and \(\lambda = 1e-5\). For testing and visualization purposes, a uniform meshgrid with size $201\times 201$ is generated. Partial experimental results are shown in Figure \ref{fig:test1}.

The exact solution is periodic in $y$ for fixed $x$, and odd in $x$ for fixed $y$, as shown in Figure~\ref{Poisson_gt}. Figure~\ref{fig:test1} shows the predictions and their corresponding absolute errors. 
As illustrated in the third row of Figure~\ref{fig:test1}, the Standard PINNs method exhibits a highly non-uniform error distribution, with the maximum absolute error—approximately 0.001—concentrated in the vicinity of \((-0.6, 0.8)\), while the error in other regions remains relatively low. In comparison, under the parameter setting of $N_H = 15, \alpha = 1/2$, and $\rho = 0.005$, the two proposed models demonstrate a notable improvement in accuracy. Specifically, the maximum absolute error is reduced to around 0.0003, corresponding to a reduction of approximately 70\% relative to the Standard PINNs.

Moreover, the error distribution in the proposed model is noticeably more uniform, which can be attributed to the regularization term introduced in our framework. Figures~\ref{poisson_standard_pinn_abs_Error} and \ref{poisson_var_abs_error} further demonstrate that proposed method consistently outperform the Standard PINNs in terms of absolute error. This indicates that the added regularization effectively constrains the model outputs, thereby mitigating large fluctuations and enhancing overall prediction stability.
For the impact on the other two parameters $\rho$ and $N_H$, the results are summarized in Table \ref{tab:2} (Benchmark: the RMSE of the Standard PINNs method is 0.00034).
  \begin{table*}[htbp]
    \centering
    \caption{Testing Error of Poisson Equation on Proposed Model($N_r = 400, \lambda = 1e-7$)}
    \label{tab:2}       % Give a unique label
    \tabcolsep=0.3cm
    \renewcommand\arraystretch{1.25}
    \begin{tabular}{c|ccc}
      \hline
      \diagbox{$\rho$}{RMSE}{$N_H$} & 15 & 20 & 30 \\
      \hline
      0.005 & 0.00018 & 0.00020 & 0.00032 \\
      % \hline
      0.01 & 0.00033 & 0.00031 & 0.00022 \\   % 注意方向参数是[]
      0.02 & 0.00013 & 0.00022 & 0.00020 \\ 
      \hline
      \end{tabular}
    \end{table*}

\subsection{Variable Coefficients Elliptic PDEs}
In this subsection, we demonstrate that the proposed method can also be applied to quasi-linear elliptic PDEs. Consider the following problem
\begin{equation}
  \left\{\begin{aligned}
  &c_1(x,y) \frac{\partial^2 u(x,y)}{\partial x^2} + c_2(x,y) \frac{\partial^2 u(x,y)}{\partial y^2} = f(x,y), \quad (x,y) \in \Omega, \\
  &u(x,y) = g(x,y), \quad (x,y) \in \partial \Omega,
  \end{aligned}\right.
  \label{quasi-linear}
\end{equation}
where \(\Omega = (-1,1)\times (-1,1)\), \(c_1(x,y) = 3 + 2 \cos \pi (x+y)\), \(c_2(x,y) = 3 + 2 \sin \pi (x+y)\), and the exact solution is given by 
\[
u(x,y) = e^{\frac{x^2}{2}} \cos(2\pi y).
\]
The source term \(f(x,y)\) and boundary function $g(x,y)$ is defined accordingly to ensure that \(u(x,y)\) satisfies Equation~\eqref{quasi-linear}.

We take the settings as layers is $[2, 20, 20, 20, 20, 1]$.  The \(\tanh(x)\) is chosen as the activation function. The parameters are set as \(N_r = 300, w_r = 1\) and \(\lambda = 1e-5\). For testing and visualization purposes, a uniform meshgrid with size $201\times 201$ is generated. 

The proposed model reduces the maximum absolute error and localizes the error regions more effectively compared to the Standard PINNs method. As shown in Figure~\ref{quasi_pinn_abs_Error}  and \ref{quasi_var_abs_error}, the maximum absolute error obtained using the Standard PINNs method is approximately 0.0008, mainly concentrated near the points \((-0.75, 0.4)\) and \((0.25, -0.6)\). In contrast, the proposed model achieves a maximum absolute error of approximately 0.0007. This corresponds to a reduction of 12.5\% compared to the Standard PINNs, along with a noticeable decrease in the spatial extent of regions exhibiting maximum errors, concentrating only around the point \((-0.75, 0.4)\).

\begin{figure*}[htbp]
  \centering
  \begin{subfigure}[b]{0.35\textwidth}
    \centering
    \includegraphics[width=\textwidth]{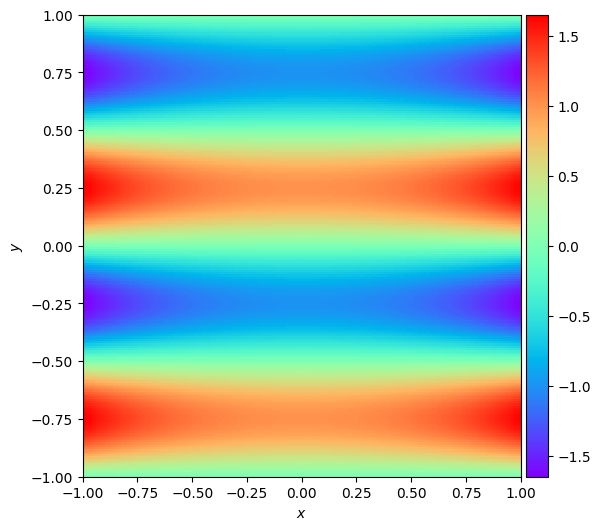}
    \caption{Ground Truth}
    \label{quasi_gt}
  \end{subfigure}\\
  \begin{subfigure}[b]{0.35\textwidth}
    \centering
    \includegraphics[width=\textwidth]{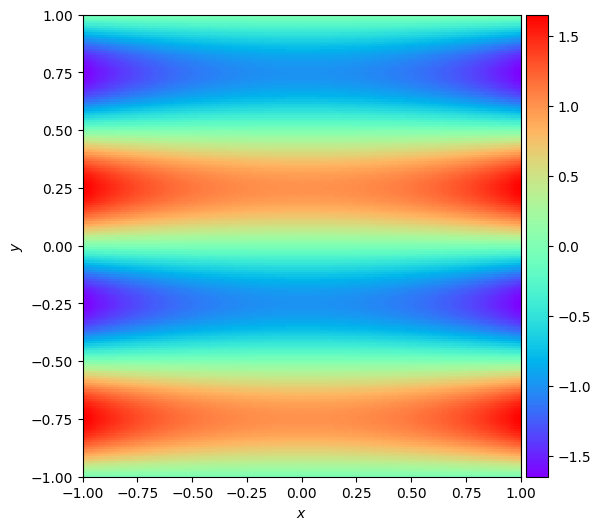}
    \caption{Standard PINNs Prediction}
    \label{quasi_pinn_pre}
  \end{subfigure}
  \begin{subfigure}[b]{0.35\textwidth}
    \centering
    \includegraphics[width=\textwidth]{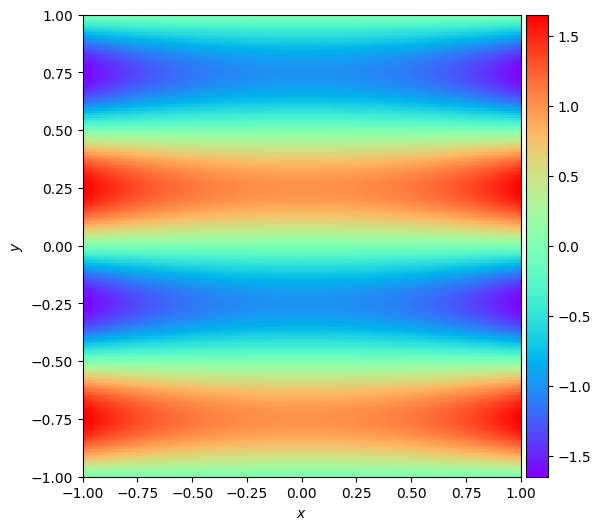}
    \caption{\centering{Ours Prediction}}
    \label{quasi_var_pre}
  \end{subfigure}

  \begin{subfigure}[b]{0.35\textwidth}
    \centering
    \includegraphics[width=\textwidth]{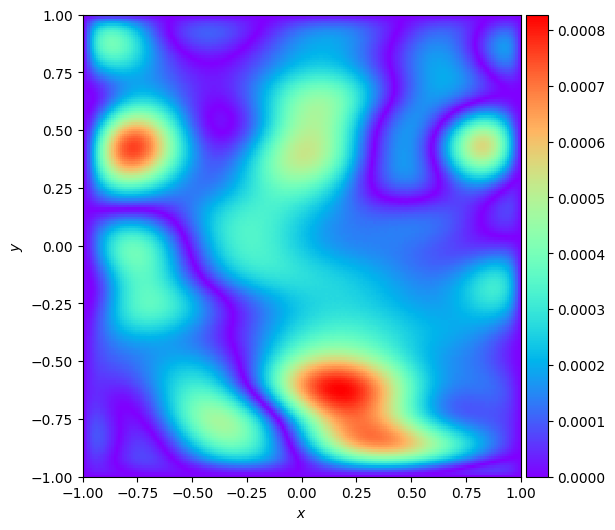}
    \caption{\centering{Standard PINNs \protect\\ Absolute Error}}
    \label{quasi_pinn_abs_Error}
  \end{subfigure}
  \begin{subfigure}[b]{0.35\textwidth}
    \centering
    \includegraphics[width=\textwidth]{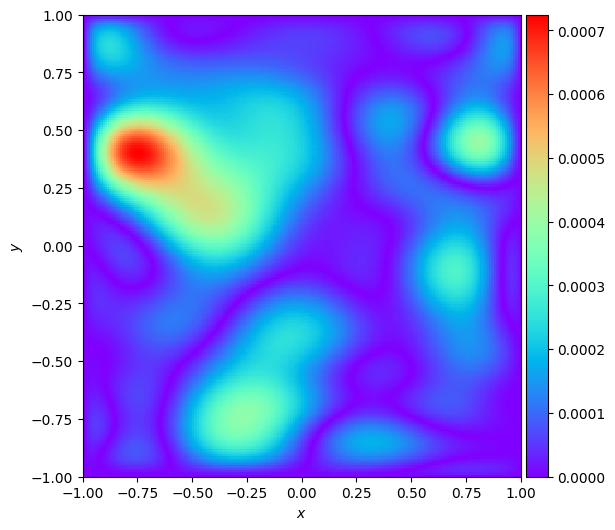}
    \caption{\centering{Ours \protect\\ Absolute Error}}
    \label{quasi_var_abs_error}
  \end{subfigure}
  \caption{The result of Variable Coefficients Elliptic PDE: $N_H=20, \alpha = 1/2, \rho = 0.01.$}
  \label{fig:test2}
\end{figure*}

The impact of the parameters $\rho$ and $N_H$ is summarized in Table \ref{tab:4}, where the RMSE of the Standard PINNs method serves as a benchmark (0.00031). The results indicate that the proposed model outperforms the Standard PINNs method when $\rho$ and $N_H$ are set to moderate values. However, performance deteriorates when these parameters are set either too high or too low, resulting in worse outcomes than the Standard PINNs. This behavior can be attributed to the fact that the exact solution exhibits good regularity with minimal variation. Thus, an excessive or insufficient number of points may cause instability in maximizing the regularization loss, ultimately compromising the accuracy of the results.
Figure \ref{fig:quasicanshubianhua} shows some visualization results of the effect of parameter $N_H$ on the experimental results.

\begin{figure*}[htbp]
  \centering
  \begin{subfigure}[b]{0.3\textwidth}
    \centering
    \includegraphics[width=\textwidth]{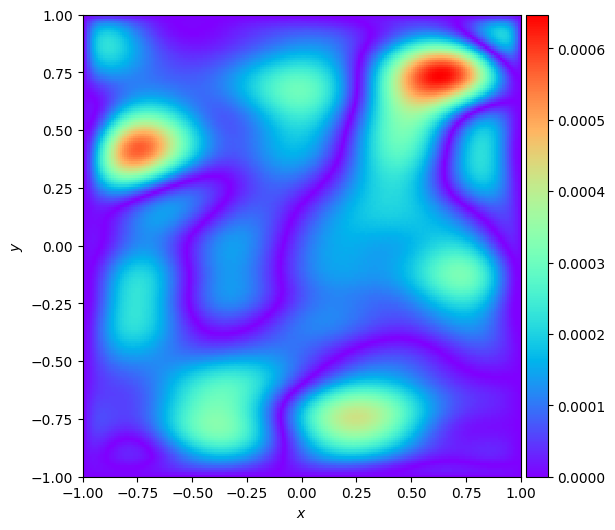}
    \caption{$N_H=15$}
    \label{quasi1510var}
  \end{subfigure}
  \begin{subfigure}[b]{0.3\textwidth}
    \centering
    \includegraphics[width=\textwidth]{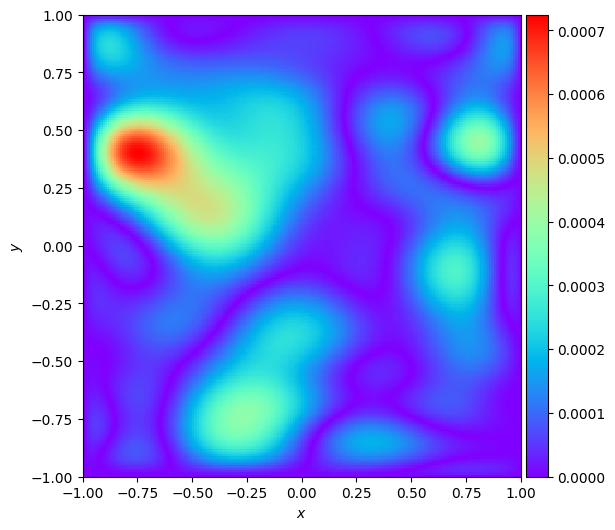}
    \caption{$N_H=20$}
    \label{quasi2010var}
  \end{subfigure}
  \begin{subfigure}[b]{0.3\textwidth}
    \centering
    \includegraphics[width=\textwidth]{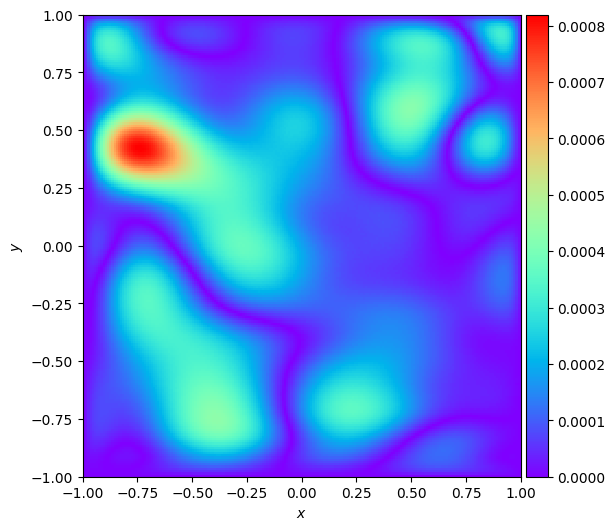}
    \caption{$N_H=30$}
    \label{quasi3010var}
  \end{subfigure}
  \caption{Impact of $N_H$ on Equation (\ref{quasi-linear})  ($N_r = 300, \alpha = 1/2, \rho = 0.01$).}
  \label{fig:quasicanshubianhua}
\end{figure*}

  \begin{table*}[htbp]
    \centering
    \caption{Testing Error of Equation (\ref{quasi-linear}) on Proposed Model($N_r = 300, \lambda = 1e-7$)}
    \label{tab:4}       % Give a unique label
    \tabcolsep=0.3cm
    \renewcommand\arraystretch{1.25}
    \begin{tabular}{c|ccc}
      \hline
      \diagbox{$\rho$}{RMSE}{$N_H$} & 15 & 20 & 30 \\
      \hline
      0.005 & 0.00043 & 0.00028 & 0.00023 \\
      % \hline
      0.01 & 0.00020 & 0.00021 & 0.00024 \\   % 注意方向参数是[]
      0.02 & 0.00026 & 0.00045 & 0.00042 \\ 
      \hline
      \end{tabular}
    \end{table*}

\subsection{Helmholtz Equation}
The Helmholtz equation is an elliptical partial differential equation that describes electromagnetic waves, and its basic form is as follows:
\begin{equation}
  \left\{\begin{aligned}
  &-(\Delta u(x,y) + k^2 u(x,y))= f(x,y) \quad (x,y)\in \Omega, \\
  &u(x,y)=g(x,y)\quad (x,y)\in \partial \Omega,
  \end{aligned}\right.
  \label{Helmholtz Equation}
\end{equation}
% -(\Delta u + k^2 u)= f,\quad \text { in } \Omega,
where $\Omega = (-1,1)\times(-1,1), k $ is a constant. we solve the Helmholtz equation with $k=\pi$, and the exact solution is specified as
\[
u(x, y) = \left(0.1 \sin (2 \pi x) + \tanh (10 x)\right) \sin (2 \pi y).
\]
The corresponding source term $f(x, y)$ and boundary condition $g(x, y)$ are derived analytically to ensure that the exact solution of Equation~(\ref{Helmholtz Equation}) is given by $u(x, y)$.

We take the settings as layers is $[2, 20, 20, 20, 20, 1]$.  The \(\tanh(x)\) is chosen as the activation function. The parameters are set as \(N_r = 500, w_r = 1\) and \(\lambda = 1e-5\). For testing and visualization purposes, a uniform meshgrid with size $201\times 201$ is generated. 

The first row of Figure~\ref{fig:test3} shows the exact solution of Equation~(\ref{Helmholtz Equation}). The second row presents the predictions obtained by training the Standard PINN method and the proposed model with parameters \(N_H = 15\), \(\alpha = 1/2\), and \(\rho = 0.02\). The corresponding absolute error distributions are illustrated in the third row.
As shown in Figure~\ref{helmholtz_pinn_abs_Error}  and \ref{helmholtz_var_abs_error}, the maximum absolute error obtained using the Standard PINNs method is approximately 0.0012, mainly concentrated near the points \((-0.25, -0.5)\). In contrast, the proposed model achieves a maximum absolute error of approximately 0.0005. This corresponds to a reduction of 58.3\% compared to the Standard PINNs.
\begin{figure*}[htbp]
  \centering
  \begin{subfigure}[b]{0.35\textwidth}
    \centering
    \includegraphics[width=\textwidth]{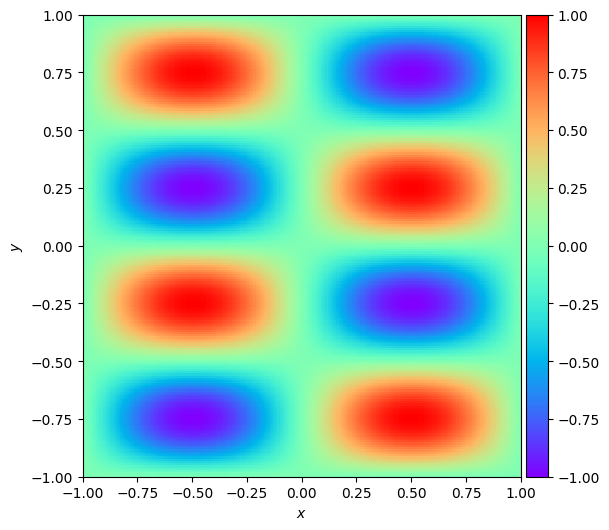}
    \caption{Ground Truth}
    \label{helmholtz_gt}
  \end{subfigure}\\
  \begin{subfigure}[b]{0.35\textwidth}
    \centering
    \includegraphics[width=\textwidth]{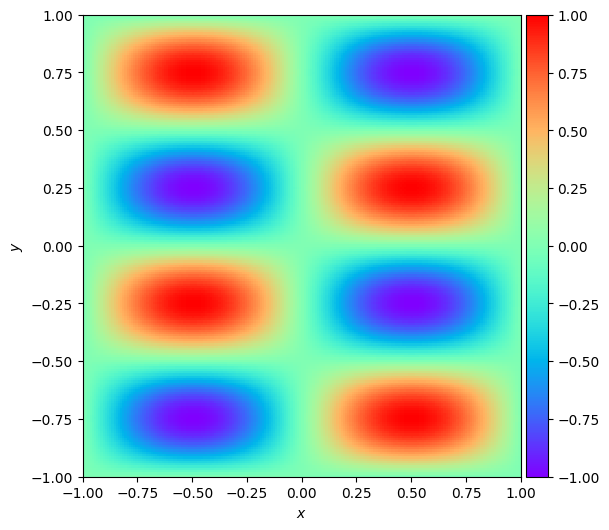}
    \caption{Standard PINNs Prediction}
    \label{helmholtz_pinn_pre}
  \end{subfigure}
  \begin{subfigure}[b]{0.35\textwidth}
    \centering
    \includegraphics[width=\textwidth]{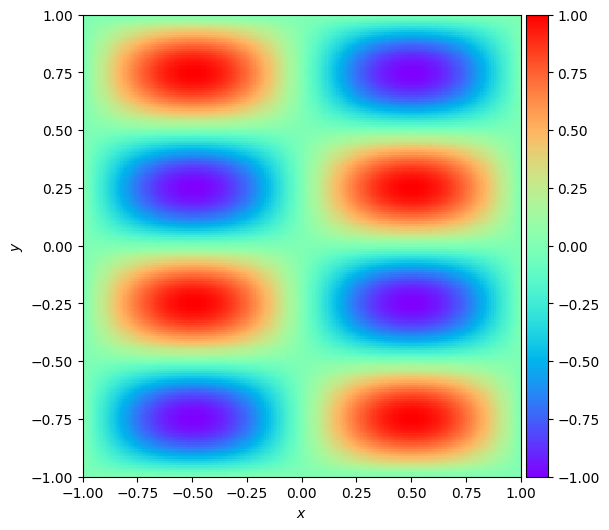}
    \caption{Ours Prediction}
    \label{helmholtz_var_pre}
  \end{subfigure}

  \begin{subfigure}[b]{0.35\textwidth}
    \centering
    \includegraphics[width=\textwidth]{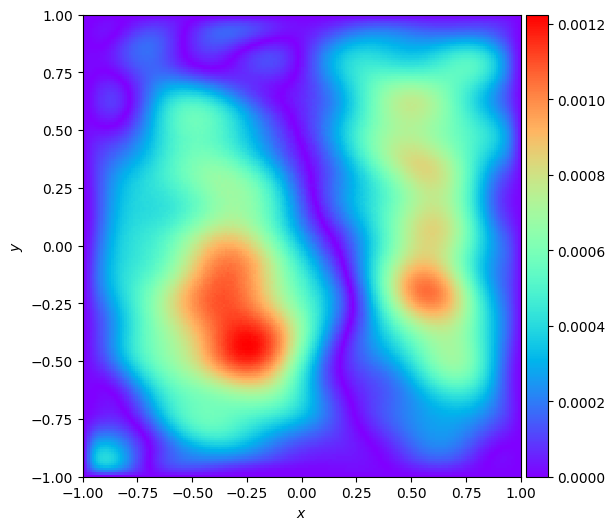}
    \caption{\centering{Standard PINNs \protect\\ Absolute Error}}
    \label{helmholtz_pinn_abs_Error}
  \end{subfigure}
  \begin{subfigure}[b]{0.35\textwidth}
    \centering
    \includegraphics[width=\textwidth]{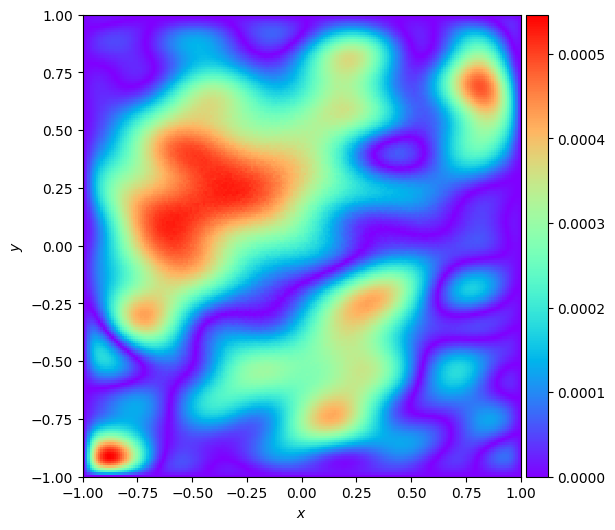}
    \caption{\centering{Ours \protect\\ Absolute Error}}
    \label{helmholtz_var_abs_error}
  \end{subfigure}
  \caption{The result of Helmholtz Equation ($N_H = 15, \rho = 0.02$). }
  \label{fig:test3}
\end{figure*}

For the impact on the other two parameters $\rho$ and $N_H$, the results are summarized in Table \ref{tab:6} (Benchmark: the RMSE of the Standard PINN method is 0.00095). 
Unlike previous cases, the proposed models exhibit higher RMSEs on the Helmholtz equation, as shown in Table \ref{tab:6}. Some failure cases are shown in Figure \ref{fig:Helmholtz failed}. 
The suboptimal performance may be attributed to two factors. First, the hyperparameter space may not have been sufficiently explored during experimentation, requiring more comprehensive optimization to identify the best configurations. Second, the network’s architectural constraints, limited neurons per layer and shallow depth, may have restricted its representational capacity, potentially causing convergence to local minima during optimization. These limitations highlight important directions for future research.

\begin{figure*}[htbp]
  \centering
  \begin{subfigure}[b]{0.3\textwidth}
    \centering
    \includegraphics[width=\textwidth]{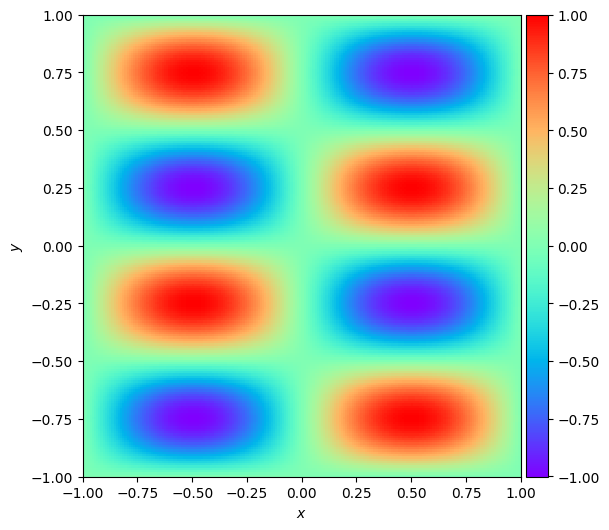}
    \captionsetup{justification=centering}
    \caption{$N_H=15, \rho = 0.01$}
    \label{helmholtz1510var_pre}
  \end{subfigure}
  \begin{subfigure}[b]{0.3\textwidth}
    \centering
    \includegraphics[width=\textwidth]{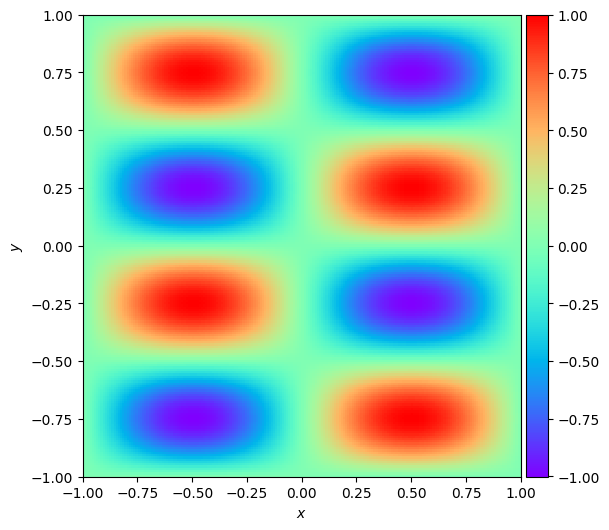}
    \captionsetup{justification=centering}
    \caption{$N_H=20, \rho = 0.01$}
    \label{helmholtz2010var_pre}
  \end{subfigure}

  \begin{subfigure}[b]{0.3\textwidth}
    \centering
    \includegraphics[width=\textwidth]{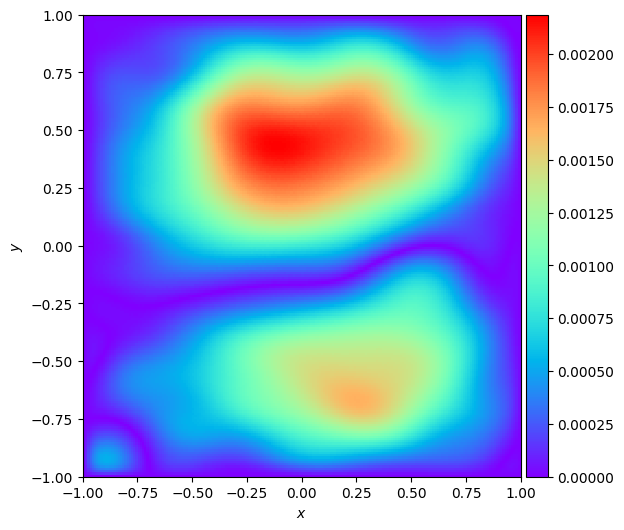}
    \caption{Absolute error of (\ref{sub@helmholtz1510var_pre})}
    \label{helmholtz1510var_error}
  \end{subfigure}
  \begin{subfigure}[b]{0.3\textwidth}
    \centering
    \includegraphics[width=\textwidth]{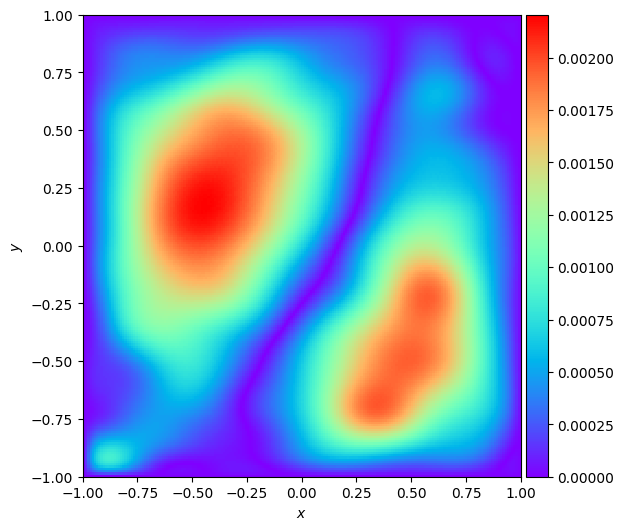}
    \caption{Absolute error of (\ref{sub@helmholtz2010var_pre})}
    \label{helmholtz2010var_error}
  \end{subfigure}
  \captionsetup{justification=centering}
  \caption{Some of the failed results.}
  \label{fig:Helmholtz failed}
\end{figure*}

  \begin{table*}[htbp]
    \centering
    \caption{Testing Error of Helmholtz Equation on Proposed Model($N_r = 500, \lambda = 1e-5$)}
    \label{tab:6}       % Give a unique label
    \tabcolsep=0.3cm
    \renewcommand\arraystretch{1.25}
    \begin{tabular}{c|ccc}
      \hline
      \diagbox{$\rho$}{RMSE}{$N_H$} & 15 & 20 & 30 \\
      \hline
      0.005 & 0.00132 & 0.00095 & 0.00111 \\
      % \hline
      0.01 & 0.00184 & 0.00195 & 0.00070 \\   % 注意方向参数是[]
      0.02 & 0.00045 & 0.00163 & 0.00134 \\ 
      \hline
      \end{tabular}
    \end{table*}

\section{Summary and Discussion}\label{Summary}
In this work, we incorporate regularity theory into neural network-based PDEs solver by introducing a novel loss function that includes the H\"older norm of the approximate solution as a regularization term. To compute the proposed regularization term, we design an approximation method and construct the corresponding model. The proposed model, along with the standard PINNs, are evaluated on several benchmark problems through extensive experiments with key hyperparameters. The results demonstrate that the proposed models achieve superior performance in solving elliptic PDEs.

Our integration of PDEs theory with neural network-based PDEs solver is not confined to elliptic partial differential equations. The framework is equally applicable to time-dependent parabolic equations and, more broadly, to any scenario where prior knowledge about the exact solution is available. Such knowledge can be effectively incorporated as a regularization term in the loss function, thereby enhancing the general applicability of the method. While this study focuses on elliptic equations as a representative case, extending the framework to other types of PDEs and incorporating more advanced theoretical insights will be the focus of future work.

\section*{CRediT authorship contribution statement}

\textbf{Qirui Zhou}: Conceptualization, Formal analysis, Methodology, Software, Validation, Visualization, Writing-original draft. \textbf{Jiebao Sun}: Conceptualization, Formal analysis, Methodology, Supervision, Validation, Funding acquisition.\textbf{Yi Ran}: Formal analysis, Investigation, Software, Validation, Writing-review and editing. \textbf{Boying Wu}: Conceptualization, Investigation, Supervision, Writing-review and editing.

\section*{Declaration of competing interest}
The authors declare that they have no known competing financial interests or personal relationships that could have appeared to influence the work reported in this paper.

\section*{Data availability}
Data will be made available on request.

\bibliographystyle{elsarticle-num} 
\bibliography{cas-refs}

\begin{thebibliography}{10}
\expandafter\ifx\csname url\endcsname\relax
  \def\url#1{\texttt{#1}}\fi
\expandafter\ifx\csname urlprefix\endcsname\relax\def\urlprefix{URL }\fi
\expandafter\ifx\csname href\endcsname\relax
  \def\href#1#2{#2} \def\path#1{#1}\fi

\bibitem{cybenko1989approximation}
G.~Cybenko, Approximation by superpositions of a sigmoidal function,
  Mathematics of control, signals and systems 2~(4) (1989) 303--314.

\bibitem{hornik1989multilayer}
K.~Hornik, M.~Stinchcombe, H.~White, Multilayer feedforward networks are
  universal approximators, Neural networks 2~(5) (1989) 359--366.

\bibitem{raissi2019physics}
M.~Raissi, P.~Perdikaris, G.~Karniadakis, Physics-informed neural networks: A
  deep learning framework for solving forward and inverse problems involving
  nonlinear partial differential equations, Journal of Computational Physics
  378 (2019) 686--707.

\bibitem{sirignano2018dgm}
J.~Sirignano, K.~Spiliopoulos, Dgm: A deep learning algorithm for solving
  partial differential equations, Journal of computational physics 375 (2018)
  1339--1364.

\bibitem{weinan2018deep}
E.~Weinan, B.~Yu, The deep ritz method: A deep learning-based numerical
  algorithm for solving variational problems, Communications in Mathematics and
  Statistics 1~(6) (2018) 1--12.

\bibitem{kharazmi2019variational}
E.~Kharazmi, Z.~Zhang, G.~E. Karniadakis, Variational physics-informed neural
  networks for solving partial differential equations, arXiv preprint
  arXiv:1912.00873 (2019).

\bibitem{jagtap2020adaptive}
A.~D. Jagtap, K.~Kawaguchi, G.~E. Karniadakis, Adaptive activation functions
  accelerate convergence in deep and physics-informed neural networks, Journal
  of Computational Physics 404 (2020) 109136.

\bibitem{yu2022gradient}
J.~Yu, L.~Lu, X.~Meng, G.~E. Karniadakis, Gradient-enhanced physics-informed
  neural networks for forward and inverse pde problems, Computer Methods in
  Applied Mechanics and Engineering 393 (2022) 114823.

\bibitem{meng2020ppinn}
X.~Meng, Z.~Li, D.~Zhang, G.~E. Karniadakis, Ppinn: Parareal physics-informed
  neural network for time-dependent pdes, Computer Methods in Applied Mechanics
  and Engineering 370 (2020) 113250.

\bibitem{lu2021deepxde}
L.~Lu, X.~Meng, Z.~Mao, G.~E. Karniadakis, Deepxde: A deep learning library for
  solving differential equations, SIAM review 63~(1) (2021) 208--228.

\bibitem{kharazmi2021hp}
E.~Kharazmi, Z.~Zhang, G.~E. Karniadakis, hp-vpinns: Variational
  physics-informed neural networks with domain decomposition, Computer Methods
  in Applied Mechanics and Engineering 374 (2021) 113547.

\bibitem{zang2020weak}
Y.~Zang, G.~Bao, X.~Ye, H.~Zhou, Weak adversarial networks for high-dimensional
  partial differential equations, Journal of Computational Physics 411 (2020)
  109409.

\bibitem{lu2021learning}
L.~Lu, P.~Jin, G.~Pang, Z.~Zhang, G.~E. Karniadakis, Learning nonlinear
  operators via deeponet based on the universal approximation theorem of
  operators, Nature machine intelligence 3~(3) (2021) 218--229.

\bibitem{li2020fourier}
Z.~Li, N.~Kovachki, K.~Azizzadenesheli, B.~Liu, K.~Bhattacharya, A.~Stuart,
  A.~Anandkumar, Fourier neural operator for parametric partial differential
  equations, arXiv preprint arXiv:2010.08895 (2020).

\bibitem{rahman2022u}
M.~A. Rahman, Z.~E. Ross, K.~Azizzadenesheli, U-no: U-shaped neural operators,
  arXiv preprint arXiv:2204.11127 (2022).

\bibitem{wen2022u}
G.~Wen, Z.~Li, K.~Azizzadenesheli, A.~Anandkumar, S.~M. Benson, U-fno—an
  enhanced fourier neural operator-based deep-learning model for multiphase
  flow, Advances in Water Resources 163 (2022) 104180.

\bibitem{goodfellow2014explaining}
I.~J. Goodfellow, J.~Shlens, C.~Szegedy, Explaining and harnessing adversarial
  examples, arXiv preprint arXiv:1412.6572 (2014).

\bibitem{madry2017towards}
A.~Madry, A.~Makelov, L.~Schmidt, D.~Tsipras, A.~Vladu, Towards deep learning
  models resistant to adversarial attacks, arXiv preprint arXiv:1706.06083
  (2017).

\bibitem{shafahi2019adversarial}
A.~Shafahi, M.~Najibi, M.~A. Ghiasi, Z.~Xu, J.~Dickerson, C.~Studer, L.~S.
  Davis, G.~Taylor, T.~Goldstein, Adversarial training for free!, Advances in
  neural information processing systems 32 (2019).

\bibitem{gouk2021regularisation}
H.~Gouk, E.~Frank, B.~Pfahringer, M.~J. Cree, Regularisation of neural networks
  by enforcing lipschitz continuity, Machine Learning 110 (2021) 393--416.

\bibitem{hoffman2019robust}
J.~Hoffman, D.~A. Roberts, S.~Yaida, Robust learning with jacobian
  regularization, arXiv preprint arXiv:1908.02729 (2019).

\bibitem{jakubovitz2018improving}
D.~Jakubovitz, R.~Giryes, Improving dnn robustness to adversarial attacks using
  jacobian regularization, in: Proceedings of the European conference on
  computer vision (ECCV), 2018, pp. 514--529.

\bibitem{tsuzuku2018lipschitz}
Y.~Tsuzuku, I.~Sato, M.~Sugiyama, Lipschitz-margin training: Scalable
  certification of perturbation invariance for deep neural networks, Advances
  in neural information processing systems 31 (2018).

\bibitem{miyato2018spectral}
T.~Miyato, T.~Kataoka, M.~Koyama, Y.~Yoshida, Spectral normalization for
  generative adversarial networks, arXiv preprint arXiv:1802.05957 (2018).

\bibitem{evans2022partial}
L.~C. Evans, Partial differential equations, Vol.~19, American Mathematical
  Society, Rhode Island, 2022.

\bibitem{krylov1996lectures}
N.~Krylov, Lectures on elliptic and parabolic equations in h$\rm\ddot{o}$lder
  spaces, American Math-ematical Society (1996).

\bibitem{mishra2023estimates}
S.~Mishra, R.~Molinaro, Estimates on the generalization error of
  physics-informed neural networks for approximating pdes, IMA Journal of
  Numerical Analysis 43~(1) (2023) 1--43.

\bibitem{lu2021physics}
L.~Lu, R.~Pestourie, W.~Yao, Z.~Wang, F.~Verdugo, S.~G. Johnson,
  Physics-informed neural networks with hard constraints for inverse design,
  SIAM Journal on Scientific Computing 43~(6) (2021) B1105--B1132.

\end{thebibliography}

\end{document}